\documentclass{amsart}
\usepackage{amssymb}
\usepackage{amsfonts}
\usepackage{amsmath}
\usepackage{amsthm, amsmath, amssymb, enumerate}
\usepackage{amssymb, enumitem}
\usepackage{bbm}
\usepackage{amscd}
\usepackage[all]{xy}
\usepackage[hidelinks]{hyperref}
\usepackage{amsmath}
\usepackage{amssymb}
\usepackage{amsthm}
\usepackage[sort&compress,numbers]{natbib}
\usepackage[left=2cm,right=2cm,bottom=3cm,top=3cm]{geometry}
\usepackage{tikz}
\usepackage{tikz-cd}
\usepackage{wrapfig}

\newtheorem{theorem}{Theorem}[section]
\newtheorem{proposition}[theorem]{Proposition}
\newtheorem{lemma}[theorem]{Lemma}
\newtheorem{corollary}[theorem]{Corollary}
\theoremstyle{definition}
\newtheorem*{claim*}{Claim}
\newtheorem{definition}[theorem]{Definition}

\AtBeginDocument{   \def\MR#1{}}

\begin{document}
\title{Projective length, phantom extensions, and the structure of torsion
modules}
\author{Martino Lupini}
\address{Dipartimento di Matematica, Universit\`{a} di Bologna\\
Bologna\\
Italy}
\email{martino.lupini@unibo.it}
\urladdr{http://www.lupini.org/}
\thanks{The author was partially supported by the Marsden Fund Fast-Start
Grant VUW1816 and a Rutherford Discovery Fellowship VUW2002
\textquotedblleft Computing the Shape of Chaos\textquotedblright\ from the
Royal Society of New Zealand, by the Starting Grant 101077154
\textquotedblleft Definable Algebraic Topology\textquotedblright\ from the
European Research Council, by the Gruppo Nazionale per le Strutture
Algebriche, Geometriche e le loro Applicazioni (GNSAGA) of the Istituto
Nazionale di Alta Matematica (INDAM), and by the University of Bologna.}
\date{\today }

\begin{abstract}
The notion of phantom extension of order a given ordinal $\alpha $ has been
introduced in collaboration with Casarosa, as an algebraic analogue of the
order of a phantom map in topology, to study the structure of flat modules.
In this companion paper we characterize phantom extension of \emph{torsion}
modules over a countable Dedekind domain $R$.\ After localizing, one can
assume that $R$ is a discrete valuation domain with maximal ideal generated
by $p\in R$. In this case, the phantom extensions of order $\alpha $ of a
countable torsion module are precisely the $p^{\omega \left( 1+\alpha
\right) }$-pure extensions introduced by Nunke in the 1960s. A module has
projective length at most $\alpha $ if and only if it is a projective object
with respect to the exact structure defined by phantom extensions of order $%
\alpha $. We prove that a countable torsion module has projective length at
most $\alpha $ if and only if it is reduced and has Ulm length at most $%
1+\alpha $, if and only if it is the colimit of a presheaf of finite torsion modules over a countable well-founded forest of rank at most $1+\alpha $.
\end{abstract}

\maketitle

\section{Introduction}

The notion of \emph{phantom extension }of order $\alpha <\omega _{1}$ has
been introduced in collaboration with Casarosa in order to study the
structure of flat modules over a Dedekind (or, more generally,\ Pr\"{u}fer)
domain. This higher-order notion of purity is inspired by a corresponding
notion of higher-order phantom map in topology. Classically, phantom maps
(of order $0$) from a locally compact second countable Hausdorff space $X$
to a countable CW complex $P$ are the continuous maps $\phi :X\rightarrow P$
whose nontriviality up to homotopy cannot be detected on compact subspaces
of $X$ \cite{mcgibbon_phantom_1995}. In other words, $\phi |_{K}$ is
nullhomotopic for every compact subspace $K$ of $X$. This is equivalent to
the assertion that $\phi $ factors up to homotopy through the quotient map $%
X\rightarrow X/K$ obtained by identifying $K$ to a point. (Precisely, these
are the so-called phantom maps \emph{of the second kind}, the first kind
being defined in terms of restrictions to finite skeleta, in the case when $%
X $ is also a polyhedron.)

The stronger notion of phantom map of order $1$ is obtained by requiring
that, for an arbitrary compact subspace, $K$ of $X$ there exist a \emph{%
phantom} map $X/K\rightarrow P$ that induces $\phi $ up to homotopy.
Recursively, one defines more generally the notion of phantom map of order $%
\alpha $ for an arbitrary ordinal, by requiring that it be induced up to
homotopy, for every compact subspace $K$ of $X$, by a phantom map $%
X/K\rightarrow P$ of order $\beta <\alpha $.

\emph{Purity} of extensions can be seen as an algebraic analogue of phantom
maps. Indeed an extension $\kappa $ of a module $C$ by another module $A$ is
pure (of order $0$) if its nontriviality cannot be detected on finite
submodules of $C$. In other words, the extension $\kappa $ is trivial when
restricted to any finite submodule $F$ of $C$. This is equivalent to the
assertion that $\kappa $ is induced by an extension of the quotient module $%
C/F$ by $A$. Agan, by requiring that $\kappa $ be induced by an extesion of $%
C/F$ by $A$ that is itself pure of order $0$ one obtains the more stringent
notion of pure extension of order $1$. Proceeding recursively in this
fashion, pure (or phantom) extensions of order $\alpha $ can be defined.

In fact, one can more generally define the notion of \emph{phantom morphism
of order }$\alpha $ in a triangulated category, as remarked in \cite[Section
8]{casarosa_projective_2025}. The notion of phantom extension is obtained as
a particular instance, when applied to the derived category of the category
of modules. Phantom extensions of \emph{flat} (i.e., torsion-free) modules
have been introduced and studied in \cite[Sections 10 and 11]%
{casarosa_projective_2025}. In this companion paper, we consider the case of 
\emph{torsion} modules. Particularly, we observe that in this case the
notion of phantom extension of order $\alpha $ is equivalent to the notion
of $p^{\omega \left( 1+\alpha \right) }$-pure extension as defined by Nunke
in \cite{nunke_purity_1963}; see also \cite%
{keef_variations_1999,keef_injective_1994,keef_generalizations_1994,irwin_primary_1993,keef_representable_1995,nunke_homology_1967,nunke_extensions_1960}%
.

This result is obtained by showing that, when $C$ is torsion, the submodule $%
\mathrm{Ph}^{\alpha }\mathrm{Ext}\left( C,A\right) $ parametrizing phantom
extensions of order $\alpha $ is equal to the Ulm submodule $u_{1+\alpha }%
\mathrm{Ext}\left( C,A\right) $ of order $1+\alpha $; see Theorem \ref%
{Theorem:ulm}. The submodule $\mathrm{Ph}^{\alpha }\mathrm{Ext}\left(
C,A\right) $ has also been characterized topologically in \cite[Theorem 8.3]%
{casarosa_projective_2025} as the \emph{Solecki submodule }$s_{\alpha }%
\mathrm{Ext}\left( C,A\right) $ defined in terms of the\emph{\ module with a
Polish cover }structure on $\mathrm{Ext}\left( C,A\right) $. These
characterizations, together with some classical homological lemmas due to
Nunke, allow us to precisely compute the \emph{complexity }(in the sense of
invariant descriptive set theory) of the classification problem for
extensions of torsion modules in terms of their Ulm invariants; see Theorem %
\ref{Theorem:Solecki-Ext} and Corollary \ref{Corollary:parametrize-Ext}.

The notion of \emph{projective length }of countable flat modules is also
introduced and studied in \cite{casarosa_projective_2025}. We consider this
notion in the torsion case: a countable torsion module $C$ has projective
length at most $\alpha $ if and only if it is a projective object in the
exact category defined by phantom extensions of order $\alpha $ or, in other
words, $\mathrm{Ph}^{\alpha }\mathrm{Ext}\left( C,-\right) =0$. We prove
that a countable torsion module has projective length at most $\alpha $ if
and only if it is reduced and has Ulm length at most $\alpha $, if and only
if it is a colimit of a presheaf of finite torsion modules over a countable
well-founded forest of rank $1+\alpha $. Recall that, if $\mathcal{C}$ is a
category, then a presheaf of finite torsion modules over $\mathcal{C}$ is a
functor from $\mathcal{C}^{\mathrm{op}}$ to the category of finite torsion
modules. In particular, this applies when $\mathcal{C}$ is an ordered set.
We regard a rooted tree as an ordered set in the usual way, by defining for
distinct nodes $x\prec y$ if and only if $y$ belongs to the path from $x$ to
the root. By a forest we mean a disjoint union of rooted trees, with the
induce order.

The rest of this paper is divided into three sections. In Section \ref%
{Section:towers} we present some lemma concerning towers of countable
modules and the functor $\mathrm{lim}^{1}$. These are used to obtain the
description of \textrm{Ph}$^{\alpha }\mathrm{Ext}$ in terms of Ulm
submodules. In Section \ref{Section:homological} we recall some homological
lemmas due to Nunke, and apply them in the case of the functors $C\mapsto
C^{\alpha }$ mapping a module to its $\alpha $-th Ulm submodule. Finally, in
Section \ref{Section:phantom} we use these lemmas to describe the Ulm
submodules of $\mathrm{Ext}\left( C,A\right) $ in terms of the Ulm
submodules of the torsion modules $C$ and $A$. This in turn gives us a
precise description of the complexity of $\mathrm{Ext}\left( C,A\right) $,
or the problem of classifying extensions of $C$ by $A$, as well as a
characterization of the countable torsion modules with a given projective
length.

For notions concerning homological algebra, derived categories, and derived
functors we refer the reader to standard textbooks such as a \cite%
{gelfand_methods_2003,mac_lane_homology_1995,kashiwara_categories_2006}, as
well as \cite[Sections 2 and 7]{casarosa_projective_2025}. For
module-theoretic terminology and background, we refer to \cite%
{rotman_introduction_2009,krylov_modules_2018} and \cite%
{fuchs_infinite_1970,fuchs_infinite_1973}. (While \cite%
{fuchs_infinite_1970,fuchs_infinite_1973} only consider abelian groups, the
terminology and results are easily generalized to modules over an arbitrary
PID.) Definitions and results from Borel complexity theory and descriptive
set theory can be found in the classical monographs \cite%
{kechris_classical_1995,gao_invariant_2009} as well as \cite[Section 12]%
{casarosa_projective_2025}. For fundamental results concerning the category
of Polish modules and pro-countable Polish modules, and the description of
their left heart in terms of modules with a Polish cover, we refer to \cite[%
Sections 4 and 5]{casarosa_projective_2025} together with \cite%
{bergfalk_definable_2024,lupini_looking_2024}.

\subsubsection*{Acknowledgments}

We are grateful to Jeffrey Bergfalk, Luigi Caputi, Nicola Carissimi, Matteo Casarosa, Dan
Christensen, Alessandro Codenotti, Ivan Di Liberti, Luisa Fiorot, Pietro
Freni, Eusebio Gardella, Alexander Kechris, Fosco Loregian, Nicholas Meadows, Andr\'{e} Nies,
Aristotelis Panagiotopoulos, Luca Reggio, Claude Schochet, and Joseph
Zielinski for many helpful comments and remarks on a preliminary version of this paper.

\section{Towers and pure extensions\label{Section:towers}}

In this section we let $R$ be a countable discrete valuation ring (DVR) that
is not a field, with maximal ideal generated by $p\in R$, and assume all
modules to be $R$-modules. We present some results about towers of torsion
modules and their $\mathrm{lim}^{1}$. We say that a module is \emph{finite }%
if it is finitely-generated. This is equivalent to the assertion that it is
a finite direct sum of cyclic modules. We denote by $K$ the field of
fractions of $R$, and set $R\left( p^{\infty }\right) :=K/R$.

For each limit ordinal $\lambda $ we let $\left( \lambda _{n}\right) $ be an
increasing sequence of countable successor ordinals convering to $\lambda $.
For a successor or zero ordinal $\alpha $ we define $\alpha _{n}=\alpha $
for every $n\in \omega $.

\subsection{Ulm length}

Given a module $G$, one defines by recursion:

\begin{itemize}
\item $p^{0}G=G$;

\item $p^{\alpha +1}G=p\left( p^{\alpha }G\right) $;

\item $p^{\lambda }G=\bigcap_{\beta <\lambda }p^{\beta }G$ for $\lambda $
limit.
\end{itemize}

For an arbitrarily ordinal $\alpha $, we set 
\begin{equation*}
u_{\alpha }G=G^{\alpha }:=p^{\omega \alpha }G\text{.}
\end{equation*}%
This is the $\alpha $-th \emph{Ulm submodule} of $G$. We also define $%
G[p^{n}]=\left\{ x\in G:p^{n}x=0\right\} $ for $n\in \omega $. We say that $%
G $ is\emph{\ bounded }if and only if $p^{n}G$ is trivial for some $n<\omega 
$. The \emph{Ulm length} of $G$ is the least $\alpha $ such that $G^{\alpha
} $ is divisible. The module $G$ is \emph{reduced }if it has no nontrivial
divisible submodule or, equivalently, $G^{\alpha }=0$ for some ordinal $%
\alpha $. When $G$ is countable and reduced, $G$ has Ulm length at most $1$
if and only if $G$ is a countable direct sum of cyclic torsion modules \cite[%
Theorem 17.3]{fuchs_infinite_1970}. Furthermore, we have that $G$ is
countable and divisible if and only if it is a countable direct sum of
copies of $R\left( p^{\infty }\right) $ \cite[Theorem 23.1]%
{fuchs_infinite_1970}.

\subsection{Towers}

Suppose that $\boldsymbol{T}=\left( T^{\left( n\right) }\right) $ is a tower
of countable modules with bonding maps $p^{\left( n,n+1\right) }:T^{\left(
n+1\right) }\rightarrow T^{\left( n\right) }$. The tower $\boldsymbol{T}$ is:

\begin{itemize}
\item \emph{reduced }if $\mathrm{lim}\boldsymbol{T}=0$;

\item \emph{monomorphic} if its bonding maps $p^{\left( n,n+1\right) }$ are
monomorphisms;

\item \emph{essentially monomorphic} if it is isomorphic to a monomorphic
tower.
\end{itemize}

The notion of (essentially) epimorphic tower is defined in similar fashion.
The \emph{derived tower }$\boldsymbol{T}_{1}$ is defined by letting $%
T_{1}^{\left( n\right) }$ be the intersection of the images of $%
p^{(n,n+k)}:T^{\left( n+k\right) }\rightarrow T^{\left( n\right) }$ for $%
k\in \omega $. More generally, one defines by recursion $T_{0}^{\left(
n\right) }=T^{\left( n\right) }$ and, for an arbitrary countable successor
ordinal $\alpha $, $T_{\alpha }^{\left( n\right) }$ to be the intersection
of the images of 
\begin{equation*}
p^{(n,n+k)}:T_{\left( \alpha -1\right) _{n+k}}^{\left( n+k\right)
}\rightarrow T^{\left( n\right) }
\end{equation*}%
for $k\in \omega $. Then $\boldsymbol{T}_{\alpha }$ is the tower $(T_{\alpha
_{n}}^{\left( n\right) })$ \cite[Section 7.2]{casarosa_projective_2025}. The
(plain) \emph{length} of a reduced tower $\boldsymbol{T}$ is the least $%
\alpha <\omega _{1}$ such that $\boldsymbol{T}_{\alpha }\cong 0$ (and $%
\boldsymbol{T}_{\alpha -1}$ is essentially monomorphic) \cite[Definition 7.3]%
{casarosa_projective_2025}. This coincides with the (plain) \emph{Solecki
length} of $\mathrm{lim}^{1}\boldsymbol{T}$ \cite[Corollary 7.8]%
{casarosa_projective_2025}. We also recall that the functor $\mathrm{lim}%
^{1} $ from reduced towers of countable modules to modules (with a Polish
cover) is faithful \cite[Proposition 9.25]{casarosa_projective_2025}.

\begin{lemma}
Let $K$ be a field, and $\boldsymbol{T}$ be a reduced tower of vector spaces
over $K$. Then $\boldsymbol{T}$ is essentially monomorphic.
\end{lemma}

\begin{proof}
For $n\in \omega $ we can write $T^{\left( n\right) }=C^{\left( n\right)
}\oplus Z^{\left( n\right) }$ where 
\begin{equation*}
Z^{\left( n\right) }:=\mathrm{\mathrm{Ker}}\left( p^{\left( n,n+1\right)
}:T^{\left( n+1\right) }\rightarrow T^{\left( n\right) }\right) \text{.}
\end{equation*}%
Thus, we have a short exact sequence of towers%
\begin{equation*}
\boldsymbol{Z}\rightarrow \boldsymbol{T}\rightarrow \boldsymbol{C}
\end{equation*}%
where 
\begin{equation*}
\mathrm{lim}^{1}\boldsymbol{Z}\cong \mathrm{\mathrm{\mathrm{\mathrm{li}}}m}{}%
\boldsymbol{Z}\cong 0\text{,}
\end{equation*}%
Considering the six-term exact sequence relating $\mathrm{lim}$ and $\mathrm{%
lim}^{1}$, we have an induced isomorphism%
\begin{equation*}
\mathrm{lim}^{1}\boldsymbol{T}\cong \mathrm{lim}^{1}\boldsymbol{C}\text{.}
\end{equation*}%
Thus, $\boldsymbol{T}$ is isomorphic to the monomorphic tower $\boldsymbol{C}
$.
\end{proof}

\begin{lemma}
\label{Lemma:plain-on-plain}Suppose that $\left\{ 0\right\} =G_{0}\subseteq
G_{1}\subseteq \cdots \subseteq G_{n}$ is a sequence of modules with a
Polish cover such that for every $i<n$, $G_{i+1}/G_{i}$ has plain length at
most $1$. Then $G_{n}$ has plain length at most $1$.
\end{lemma}

\begin{proof}
This follows from \cite[Lemma 13.13]{casarosa_projective_2025}(2) by
induction on $n\in \omega $.
\end{proof}

\begin{lemma}
\label{Lemma:uniformly-bounded-tower}Suppose that $\ell \in \omega $ and $%
\boldsymbol{T}$ is a reduced tower of $R/p^{\ell }R$-modules. Then $%
\boldsymbol{T}$ has plain length at most $1$.
\end{lemma}

\begin{proof}
By induction on $\ell $. For $\ell =1$ this follows from the case of vector
spaces. Suppose that the conclusion holds for $\ell $. Let $\boldsymbol{T}$
be a reduced tower of $R/p^{\left( \ell +1\right) }R$-modules. Then we have
a short exact sequence of towers%
\begin{equation*}
p^{\ell }\boldsymbol{T}\rightarrow \boldsymbol{T}\rightarrow \boldsymbol{T}%
/p^{\ell }\boldsymbol{T}\text{.}
\end{equation*}%
By the inductive hypothesis, we have that $\mathrm{lim}^{1}\boldsymbol{T}%
/p^{\ell }\boldsymbol{T}$ has plain rank at most $\ell $. Furthermore, since 
$p^{\ell }\boldsymbol{T}$ is a tower of $R/pR$-vector spaces, $\mathrm{lim}%
^{1}p^{\ell }\boldsymbol{T}$ has plain rank at most $1$. It follows from
Lemma \ref{Lemma:plain-on-plain}, \cite[Lemma 5.9]{casarosa_projective_2025}%
, and the six-term exact sequence relating $\mathrm{lim}$ and $\mathrm{lim}%
^{1}$, that $\mathrm{lim}^{1}\boldsymbol{T}$ has plain length at most $1$.
\end{proof}

\begin{lemma}
Let $\boldsymbol{T}$ be a reduced tower of countable bounded modules. Fix $%
k\in \omega $. Then $p^{k}(\mathrm{lim}^{1}\boldsymbol{T})$ is $\boldsymbol{%
\Sigma }_{2}^{0}$ in $\mathrm{lim}^{1}\boldsymbol{T}$.
\end{lemma}

\begin{proof}
We have that $p^{k}\mathrm{lim}^{1}\boldsymbol{T}$ is the image of $\mathrm{%
lim}^{1}\boldsymbol{T}$ under the homomorphism induced by the morphism of
towers 
\begin{equation*}
\boldsymbol{T}\rightarrow \boldsymbol{T},\left( x_{n}\right) \mapsto \left(
p^{k}x_{n}\right) .
\end{equation*}%
\ This is also the image of 
\begin{equation*}
\mathrm{lim}_{n}^{1}\frac{T^{\left( n\right) }}{T^{\left( n\right) }[p^{k}]}
\end{equation*}%
under the morphism of towers 
\begin{equation*}
\varphi :\left( \frac{T^{\left( n\right) }}{T^{\left( n\right) }[p^{k}]}%
\right) _{n\in \omega }\rightarrow \boldsymbol{T}\text{, }\left(
x_{n}+T^{\left( n\right) }[p^{k}]\right) _{n\in \omega }\mapsto \left(
p^{k}x_{n}\right) _{n\in \omega }\text{.}
\end{equation*}%
Consider the short exact sequence of towers%
\begin{equation*}
0\rightarrow \left( \frac{T^{\left( n\right) }}{T^{\left( n\right) }[p^{k}]}%
\right) _{n\in \omega }\overset{\varphi }{\rightarrow }\boldsymbol{T}%
\rightarrow \left( \frac{T^{\left( n\right) }}{p^{k}T^{\left( n\right) }}%
\right) _{n\in \omega }\rightarrow 0\text{.}
\end{equation*}%
This induces an exact sequence%
\begin{equation*}
\mathrm{lim}_{n}^{1}\frac{T^{\left( n\right) }}{T^{\left( n\right) }[p^{k}]}%
\rightarrow \mathrm{lim}^{1}\boldsymbol{T}\rightarrow \mathrm{lim}_{n}^{1}%
\frac{T^{\left( n\right) }}{p^{k}T^{\left( n\right) }}\rightarrow 0\text{.}
\end{equation*}%
Since%
\begin{equation*}
\left( \frac{T^{\left( n\right) }}{p^{k}T^{\left( n\right) }}\right) _{n\in
\omega }
\end{equation*}%
is a tower of $R/p^{k}R$-modules, it has plain length at most $1$. Thus, the
image of 
\begin{equation*}
\mathrm{lim}_{n}^{1}\frac{T^{\left( n\right) }}{T^{\left( n\right) }[p^{k}]}%
\rightarrow \mathrm{lim}^{1}\boldsymbol{T}
\end{equation*}%
is $\boldsymbol{\Sigma }_{2}^{0}$ in $\mathrm{lim}^{1}\boldsymbol{T}$ by 
\cite[Lemma 5.9]{casarosa_projective_2025}, and the conclusion follows.
\end{proof}

\begin{lemma}
Let $\boldsymbol{T}$ be a reduced tower of countable bounded modules. Then 
\begin{equation*}
u_{1}\mathrm{lim}^{1}\boldsymbol{T}=s_{1}\mathrm{lim}^{1}\boldsymbol{T}\text{%
.}
\end{equation*}
\end{lemma}

\begin{proof}
Suppose that $\left( a_{i}\right) $ represents an element of $u_{1}\mathrm{%
lim}^{1}\boldsymbol{T}$. Fix $n\in \omega $.\ Let $k\in \omega $ be such
that $p^{k}T^{\left( i\right) }=0$ for $i\leq n$. Since $\left( a_{n}\right) 
$ represents an element of $u_{1}\mathrm{lim}^{1}\boldsymbol{T}$ there exist 
$c_{i},b_{i}\in T^{\left( i\right) }$ such that 
\begin{equation*}
c_{i}-p^{\left( i,i+1\right) }c_{i+1}+p^{k}b_{i}=a_{i}
\end{equation*}%
for $i\in \omega $. In particular, 
\begin{equation*}
c_{i}-p^{\left( i,i+1\right) }c_{i+1}=a_{i}
\end{equation*}%
for $i\leq n$. By \cite[Lemma 7.10]{casarosa_projective_2025} this shows
that $\left( a_{i}\right) $ represents an element of $s_{1}\mathrm{lim}^{1}%
\boldsymbol{T}$. Thus, $u_{1}\mathrm{lim}^{1}\boldsymbol{T}$ is contained in 
$s_{1}\mathrm{lim}^{1}\boldsymbol{T}$. In order to establish the other
inclusion, since $s_{1}\mathrm{lim}^{1}\boldsymbol{T}$ is the smallest $%
\boldsymbol{\Pi }_{3}^{0}$ submodule with a Polish cover of $\mathrm{lim}^{1}%
\boldsymbol{T}$, it suffices to prove that $u_{1}\mathrm{lim}^{1}\boldsymbol{%
T}$ is $\boldsymbol{\Pi }_{3}^{0}$ in $\mathrm{\mathrm{li}m}^{1}\boldsymbol{T%
}$.\ Since%
\begin{equation*}
u_{1}\mathrm{lim}^{1}\boldsymbol{T}=\bigcap_{k\in \omega }p^{k}\mathrm{lim}%
^{1}\boldsymbol{T}\text{,}
\end{equation*}%
the conclusion follows from the previous lemma.
\end{proof}

\begin{theorem}
\label{Theorem:towers}Let $\boldsymbol{T}$ be a tower of countable \emph{%
bounded} modules. Fix $\alpha <\omega _{1}$. Then 
\begin{equation*}
u_{\alpha }\mathrm{lim}^{1}\boldsymbol{T}=s_{\alpha }\mathrm{lim}^{1}%
\boldsymbol{T}
\end{equation*}%
and the inclusion $\boldsymbol{T}_{\alpha }\subseteq \boldsymbol{T}$ induces
an isomorphism%
\begin{equation*}
\mathrm{lim}^{1}\boldsymbol{T}_{\alpha }\cong s_{\alpha }\mathrm{lim}^{1}%
\boldsymbol{T}\text{.}
\end{equation*}
\end{theorem}

\begin{proof}
Without loss of generality we can assume that $\boldsymbol{T}$ is reduced.
The second assertion is established in \cite[Theorem 7.6]%
{casarosa_projective_2025} when $\boldsymbol{T}$ is an arbitrary tower of
countable modules. We prove that the first assertion holds by induction on $%
\alpha $.\ When $\alpha $ is limit, the conclusion follows from the
inductive hypothesis. Let us thus assume that $\alpha $ is a countable
successor ordinal. Considering the isomorphism%
\begin{equation*}
s_{\alpha -1}\mathrm{lim}^{1}\boldsymbol{T}\cong \mathrm{lim}^{1}\boldsymbol{%
T}_{\alpha -1}
\end{equation*}%
after replacing $\boldsymbol{T}$ with $\boldsymbol{T}_{\alpha -1}$ we can
assume without loss of generality that $\alpha =1$. In this case, the
conclusion follows from the previous lemma.
\end{proof}

\subsection{Pure extensions\label{Section:pure}}

Let $C,A$ be countable modules. The submodule $\mathrm{PExt}\left(
C,A\right) $ of $\mathrm{Ext}\left( C,A\right) $ is defined to be the
intersection of%
\begin{equation*}
\mathrm{Ran}\left( \mathrm{Ext}\left( C/C_{0},A\right) \rightarrow \mathrm{%
Ext}\left( C,A\right) \right) =\mathrm{\mathrm{\mathrm{\mathrm{Ker}}}}\left( 
\mathrm{Ext}\left( C,A\right) \rightarrow \mathrm{Ext}\left( C_{0},A\right)
\right)
\end{equation*}%
where $C_{0}$ ranges among the finite submodules of $C$. By definition, $%
\mathrm{PExt}\left( C,A\right) $ parametrizes the extensions of $C$ by $A$
that are \emph{pure }\cite[Section 3.3.1]{rotman_introduction_2009}.
Equivalently, $\mathrm{PExt}\left( C,A\right) $ can be seen as the group $%
\mathrm{Ext}_{\mathcal{E}_{0}}\left( C,A\right) $ of admissible short exact
sequences with respect to the exact structure \emph{projectively generated}
by \emph{finite} modules \cite[Theorem 3.69]{rotman_introduction_2009}. This
coincides with the exact structure projectively generated by \emph{bounded }%
modules.

If $C=\mathrm{co\mathrm{lim}}_{n}C_{n}$, where $C_{n}$ is a bounded module
for every $n\in \omega $, then%
\begin{equation*}
\mathrm{PExt}\left( C,A\right) \cong \mathrm{lim}_{n}^{1}\mathrm{Hom}\left(
C_{n},A\right) \text{;}
\end{equation*}%
see \cite[Theorem 8.3]{casarosa_projective_2025}. By \cite[Theorem 5.3]%
{fuchs_infinite_1973}, $\mathrm{PExt}\left( C,A\right) $ coincides with the
first Ulm submodule $u_{1}\mathrm{Ext}\left( C,A\right) $. We also have that 
$\mathrm{PExt}\left( C,A\right) $ is the closure of the trivial submodule in 
$\mathrm{Ext}\left( C,A\right) $, namely the Solecki submodule $s_{0}\mathrm{%
Ext}\left( C,A\right) $ \cite[Corollary 11.6]{eilenberg_group_1942}.

When $C$ is torsion and $A$ is flat (i.e., torsion-free), $\mathrm{Ext}%
\left( C,A\right) $ is Polish, and $\mathrm{PExt}\left( C,A\right) =0$ \cite[%
Section 9.1]{casarosa_projective_2025}. If $A$ is an arbitrary countable
modules with largest torsion submodule $A_{\mathrm{t}}$, 
\begin{equation*}
\mathrm{Hom}\left( C,A/A_{\mathrm{t}}\right) \cong \mathrm{PExt}\left(
C,A/A_{\mathrm{t}}\right) \cong 0\text{.}
\end{equation*}%
Thus, the inclusion $A_{\mathrm{t}}\rightarrow A$ induces an isomorphism%
\begin{equation*}
\mathrm{PExt}\left( C,A_{\mathrm{t}}\right) \cong \mathrm{PExt}\left(
C,A\right) \text{.}
\end{equation*}%
Thus, when $C$ is torsion, when studying pure extensions of $C$ by $A$ we
can assume without loss of generality that also $A$ is torsion. We can also
assume without loss of generality that $A$ is \emph{reduced}. Indeed,
replacing $A$ with the quotient of $A$ by its largest divisible submodule $%
\Delta A$ does not change $\mathrm{Ext}\left( C,A\right) $, by virtue of
injectivity of divisible modules.

\subsection{Higher order pure extensions}

Phantom extensions of order $\alpha <\omega _{1}$ have been introduced in 
\cite[Section 8.3, Section 9.1]{casarosa_projective_2025} as a
generalization of pure extension, which are the phantom extensions of order $%
0$. Phantom extensions of $C$ by $A$ of order $\alpha $ parametrize a
submodule $\mathrm{Ph}^{\alpha }\mathrm{Ext}\left( C,A\right) $ of $\mathrm{%
Ext}\left( C,A\right) $ defined recursively as follows: 
\begin{equation*}
\mathrm{Ph}^{0}\mathrm{Ext}\left( C,A\right) =\mathrm{PExt}\left( C,A\right) 
\text{,}
\end{equation*}%
and $\mathrm{Ph}^{\alpha }\mathrm{Ext}\left( C,A\right) $ is the
intersection of%
\begin{equation*}
\mathrm{Ran}\left( \mathrm{Ph}^{\beta }\mathrm{Ext}\left( C/C_{0},A\right)
\rightarrow \mathrm{Ext}\left( C,A\right) \right)
\end{equation*}%
for $\beta <\alpha $ and $C_{0}\subseteq C$ finite. By \cite[Theorem 8.3]%
{casarosa_projective_2025}, $\mathrm{Ph}^{\alpha }\mathrm{Ext}\left(
C,A\right) $ is equal to the Solecki submodule $s_{\alpha }\mathrm{Ext}%
\left( C,A\right) $.

Another generalization of the notion of pure extension was introduced in 
\cite{nunke_purity_1963}; see also \cite[Chapter XII]{fuchs_infinite_1973}.
An extension of $C$ by $A$ is $p^{\omega \alpha }$-pure for some ordinal $%
\alpha $ if it represents an element of $u_{\alpha }\mathrm{Ext}\left(
C,A\right) $. Thus, purity is obtained as a particular case when $\alpha =1$%
. The following result reconciles these higher-order notions of purity.

\begin{theorem}
\label{Theorem:ulm}Suppose that $C$ and $A$ are countable modules, with $C$
torsion.For $n\in \omega $ define%
\begin{equation*}
T^{\left( n\right) }:=\mathrm{Hom}\left( C[p^{n}],A\right)
\end{equation*}%
and let $\boldsymbol{T}$ be the tower $\left( T^{\left( n\right) }\right) $.
Then%
\begin{equation*}
\mathrm{PExt}\left( C,A\right) \cong \mathrm{lim}_{n}^{1}\boldsymbol{T}
\end{equation*}%
and for every $\alpha <\omega _{1}$, 
\begin{equation*}
\mathrm{Ph}^{\alpha }\mathrm{Ext}\left( C,A\right) =s_{\alpha }\mathrm{PExt}%
\left( C,A\right) =u_{\alpha }\mathrm{PExt}\left( C,A\right) =u_{1+\alpha }%
\mathrm{Ext}\left( C,A\right) \cong \mathrm{lim}_{n}^{1}\boldsymbol{T}%
_{\alpha }\text{.}
\end{equation*}%
Thus, an extension of $C$ by $A$ is phantom of order $\alpha $ if and only
if it is $p^{\omega \left( 1+\alpha \right) }$-pure.
\end{theorem}

\begin{proof}
This is an immediate consequence of the above remarks and Theorem \ref%
{Theorem:towers}.
\end{proof}

An extension 
\begin{equation*}
\mathfrak{S}:A\rightarrow X\overset{\pi }{\rightarrow }C
\end{equation*}
is pure if and only if an element $c$ of $C$ admits a lift $\hat{c}$ in $X$
of the same order. One can characterize phantom extensions of order $\alpha $
in a similar fashion, showcasing how the notion of phantom extension of
order $\alpha $ can be seen as a higher order version of purity; see \cite[%
Proposition 9.8]{casarosa_projective_2025}.

Let us say that a triple%
\begin{equation*}
(c,C_{0},\hat{C}_{0})
\end{equation*}%
where $c\in C$, $C_{0}$ is a submodule of $C$, and $\hat{C}_{0}$ is a
submodule of $X$ with $\pi \hat{C}_{0}=C_{0}$, has:

\begin{itemize}
\item a \emph{sharp lift of order }$0$ if there exists a lift $\hat{c}\in X$
of $c$ whose order modulo $\hat{C}_{0}$ is equal to the order of $c$ modulo $%
C_{0}$, i.e., the order of $\hat{c}+\hat{C}_{0}$ in the quotient $X/\hat{C}%
_{0}$ is equal to the order of $c+C_{0}$ in the quotint $C/C_{0}$;

\item a \emph{sharp lift of order }$\alpha >0$ if for every $\beta <\alpha $
there exists $\hat{c}\in X$ of $c$ that whose order modulo $\hat{C}_{0}$ is
equal to the order of $c$ modulo $C_{0}$, and for every $d\in C_{0}$, 
\begin{equation*}
(d_{0},\left\langle C_{0},c\right\rangle ,\left\langle \hat{C}_{0},\hat{c}%
\right\rangle )
\end{equation*}%
has a sharp lift of order $\beta $.
\end{itemize}

By definition, the extension $\mathfrak{S}$ is pure of order $\alpha $ if
and only if for every $c\in C_{0}$, the triple $\left( c,\left\{ 0\right\}
,\left\{ 0\right\} \right) $ has a sharp lift of order $\alpha $. It is
transparent that this recovers the usual notion of purity for $\alpha =0$.
It is proved in \cite[Proposition 9.8]{casarosa_projective_2025} that an
extension is pure of order $\alpha $ if and only if it is phantom of order $%
\alpha $.

\subsection{Projective length\label{Subsection:projective-length}}

Recall that a countable torsion module $C$ is \emph{pure-projective }\cite[%
Section 30]{fuchs_infinite_1970}\emph{\ }if $\mathrm{PExt}\left( C,A\right)
=0$ for any countable torsion module $A$. This is equivalent to the
assertion that $C$ is a countable direct sum of cyclic modules \cite[Theorem
30.2]{fuchs_infinite_1970}, as well as to the fact that it has Ulm length at
most $1$.

Generalizing this notion, we say that a countable torsion module $C$ is has
projective length at most $\alpha $, for some countable ordinal $\alpha $,
if $\mathrm{Ph}^{\alpha }\mathrm{Ext}\left( C,A\right) =0$ for every
countable torsion module $A$; see \cite[Section 9.9]%
{casarosa_projective_2025}.\ When $\alpha $ is a successor, we say that $C$
has \emph{plain projective length }at most $\alpha $ if $\left\{ 0\right\}
\in \boldsymbol{\Sigma }_{2}^{0}(\mathrm{Ph}^{\alpha -1}\mathrm{Ext}\left(
C,A\right) )$ for every countable torsion module $A$.

We recall the definition of the \textquotedblleft
universal\textquotedblright\ well-founded tree $I_{\alpha }^{\mathrm{plain}}$
of rank $\alpha $ for a successor ordinal $\alpha <\omega _{1}$, and of the
well-founded forest $I_{\alpha }$ of rank $\alpha $ for $\alpha <\omega _{1}$%
. We regard rooted trees as ordered set with respect to the relation between
nodes defined by setting $x\leq y$ if and only if $x$ belongs to the unique
path from $y$ to the root. By a forest we mean a disjoint union of rooted
trees, with the induced order. One sets $I_{1}^{\mathrm{plain}}=\left\{
0\right\} $. If $I_{\beta }^{\mathrm{plain}}$ and $I_{\beta }$ have been
defined for $\beta <\alpha $, one lets $I_{\alpha }^{\mathrm{plain}}$ to be $%
\left\{ \alpha -1\right\} \cup I_{\alpha -1}$, where $\alpha -1$ is the root
and the inclusion $I_{\alpha -1}\rightarrow I_{\alpha }^{\mathrm{plain}}$ is
order-preserving. One also lets $I_{\alpha }$ to be the disjoint union of $%
I_{\alpha _{n}}^{\mathrm{plain}}$ for $n\in \omega $. It is easily seen by
induction on $\alpha $ that $I_{\alpha }^{\mathrm{plain}}$ is a
\textquotedblleft universal\textquotedblright\ well-founded tree of rank at
most $\alpha $, in the sense that for every well-founded tree of rank at
most $\alpha $ is isomorphic to a subtree of $I_{\alpha }^{\mathrm{plain}}$
closed under initial segments. Likewise, $I_{\alpha }$ is a universal
well-founded forest of rank at most $\alpha $.

For $\alpha <\omega _{1}$, let $\mathcal{S}_{\alpha }$ be the class of
countable torsion modules of plain projective length at most $\alpha $. We
define $\mathcal{E}_{\alpha }$ to be the exact structure on the class $%
\mathbf{Tor}\left( R\right) $ of torsion modules projectively generated by $%
\mathcal{S}_{\alpha _{n}}$ for $n<\omega $. The arguments in \cite[Section 10%
]{casarosa_projective_2025} apply verbatim to the context of torsion
modules, yielding the following characterization of countable torsion
modules of projective length at most $\alpha $.

\begin{proposition}
\label{Proposition:projective-length}Let $C$ be a countable torsion module,
and $\alpha <\omega _{1}$. The following assertions are equivalent:

\begin{enumerate}
\item $C$ has projective length at most $\alpha $;

\item $C$ is a direct summand of a colimit of a presheaf of finite torsion
modules over $I_{1+\alpha }$;

\item $C$ is a direct summand of a colimit of a presheaf of finite torsion
modules over a countable well-founded forest of rank at most $1+\alpha $;

\item $C$ is $\mathcal{E}_{\alpha }$-projective.
\end{enumerate}
\end{proposition}

The same arguments as in \cite[Section 10]{casarosa_projective_2025} also
produce projective resolutions in the exact category $\left( \mathbf{Tor}%
\left( R\right) ,\mathcal{E}_{\alpha }\right) $.

\begin{theorem}
\label{Theorem:projective-length}Fix $\alpha <\omega _{1}$. Then the exact
category $\left( \mathbf{Tor}\left( R\right) ,\mathcal{E}_{\alpha }\right) $
is hereditary with enough projectives. The functor $\mathrm{Ph}^{\alpha }%
\mathrm{Ext}$ on $\mathbf{Tor}\left( R\right) $ is the first cohomological
right derived functor $\mathrm{Ext}_{\mathcal{E}_{\alpha }}^{1}$ of $\mathrm{%
Hom}$ in the category $\left( \mathbf{Tor}\left( R\right) ,\mathcal{E}%
_{\alpha }\right) $.
\end{theorem}

\section{Homological lemmas\label{Section:homological}}

In this section we continue to assume that $R$ is a countable DVR and all
modules are $R$-modules. We recall some classical homological terminology
and lemmas, particularly from the seminal paper \cite{nunke_homology_1967},
to be used in the proofs of our main results.

\subsection{Radicals and cotorsion functors\label{Subsection:cotorsion}}

A \emph{preradical }or \emph{subfunctor of the identity }$S$ is a function $%
A\mapsto SA$ assigning to each module $A$ a submodule $SA\subseteq A$ such
that, if $f:A\rightarrow B$ is a homomorphism, then $f$ maps $SA$ to $SB$.
This gives a functor from the category of modules to itself, where one
defines for a homomorphism $f:A\rightarrow B$, $Sf:=f|_{SA}:SA\rightarrow SB$%
. A \emph{radical }is a preradical $S$ such that $S(A/SA)=0$ for every group 
$A$.

An extension $R\rightarrow G\rightarrow H$ of a \emph{countable} module $H$
by $R$ defines a preradical $S$, by setting%
\begin{equation*}
SA=\mathrm{Ran}\left( \mathrm{Hom}\left( G,A\right) \rightarrow \mathrm{Hom}%
\left( R,A\right) =A\right) =\mathrm{\mathrm{Ker}}\left( A=\mathrm{Hom}%
\left( R,A\right) \rightarrow \mathrm{Ext}\left( H,A\right) \right) \text{.}
\end{equation*}%
In this case, one says that $S$ is the preradical \emph{represented} by the
extension $R\rightarrow G\rightarrow H$. A \emph{cotorsion functor }is a
preradical that is represented by an extension $R\rightarrow G\rightarrow H$
such that $H$ is a countable torsion module. Notice that, as $H$ is
countable, this implies that $SA$ is a submodule with a Polish cover of $A$
whenever $A$ is a module with a Polish cover.

Suppose that $S$ is a cotorsion functor. An extension $C\rightarrow
E\rightarrow A$ is $S$-pure if it defines an element of $S\mathrm{Ext}\left(
A,C\right) $. In this case, the map $C\rightarrow E$ is an $S$-pure
injective homomorphism and the map $E\rightarrow A$ is an $S$-pure
surjective homomorphism, and the image of $C$ inside of $E$ is an $S$-pure
subgroup. A module $A$ is $S$-projective if $S\mathrm{Ext}\left( A,C\right)
=0$ for every group $C$, and $S$-injective if $S\mathrm{Ext}\left(
C,A\right) =0$ for every group $C$. The cotorsion functor $S$ \emph{has
enough projectives }if for every module $A$ there exists an $S$-pure
extension $M\rightarrow P\rightarrow A$ where $P$ is $S$-projective. This is
equivalent to the assertion that $S$ is represented by an extension $%
R\rightarrow G\rightarrow H$ where $H$ is an $S$-projective torsion module 
\cite[Theorem 4.8]{nunke_purity_1963}. A cotorsion functor with enough
projectives is necessarily a radical.

The following lemma is a consequence of \cite[Lemma 1.1]{nunke_homology_1967}%
.

\begin{lemma}
\label{Lemma:exact-functor}Suppose that $S$ is a cotorsion functor, $A,C$
are countable modules, and $B$ is a submodule of $C$ contained in $SC$. Then
the exact sequence $B\rightarrow C\rightarrow C/B$ induces an exact sequence%
\begin{equation*}
0\rightarrow \mathrm{Hom}\left( A,B\right) \rightarrow \mathrm{Hom}%
(A,C)\rightarrow \mathrm{Hom}\left( A,C/B\right) \rightarrow \mathrm{Ext}%
\left( A,B\right) \rightarrow S\mathrm{Ext}(A,C)\rightarrow S\mathrm{Ext}%
\left( A,C/B\right) \rightarrow 0
\end{equation*}%
of modules with a Polish cover. Furthermore, $S\mathrm{Ext}(A,C)$ is the
preimage of $S\mathrm{Ext}\left( A,C/B\right) $ under the surjective
homomorphism $\mathrm{Ext}\left( A,C\right) \rightarrow \mathrm{Ext}\left(
A,C/B\right) $.
\end{lemma}

Suppose that $S$ is a cotorsion functor represented by the exact sequence $%
R\rightarrow G_{S}\rightarrow H_{S}$, where $H_{S}$ is a countable $S$%
-projective torsion module. If $C$ is a countable module, then we have a
corresponding exact sequence of modules with a Polish cover%
\begin{equation*}
0\rightarrow SC\rightarrow C\rightarrow \mathrm{Ext}\left( H_{S},C\right)
\rightarrow \mathrm{Ext}\left( G_{S},C\right) \rightarrow 0\text{.}
\end{equation*}%
The proofs of \cite[Lemma 1.3, Lemma 1.4, and Theorem 1.11]%
{nunke_homology_1967} give the following lemma.

\begin{lemma}
\label{Lemma:short-exact-functor}Suppose that $S$ is a cotorsion functor
represented by an extension $R\rightarrow G_{S}\rightarrow H_{S}$ where $%
H_{S}$ is a countable $S$-projective torsion module. Let $A,C$ be a
countable modules such that $A/SA$ is $S$-projective. Then:

\begin{enumerate}
\item The quotient map $C\rightarrow C/SC$ induces an isomorphisms 
\begin{equation*}
\mathrm{Ext}\left( G_{S},C\right) \rightarrow \mathrm{Ext}\left(
G_{S},C/SC\right)
\end{equation*}
and 
\begin{equation*}
\mathrm{Ext}\left( H_{S},C\right) \rightarrow \mathrm{Ext}\left(
H_{S},C/SC\right) \text{;}
\end{equation*}

\item If $SC=0$ then $S\mathrm{Ext}\left( A,C\right) $ and $\mathrm{Hom}%
\left( SA,\mathrm{Ext}\left( G_{S},C\right) \right) =\mathrm{Hom}\left( SA,S%
\mathrm{Ext}\left( G_{S},C\right) \right) $ are isomorphic Polish modules;

\item The short exact sequences $SA\rightarrow A\rightarrow A/SA$ and $%
SC\rightarrow C\rightarrow C/SC$ induce homomorphisms 
\begin{equation*}
\mathrm{Hom}(A,C/SC)\rightarrow \mathrm{Ext}(A,SC)
\end{equation*}%
and 
\begin{equation*}
\mathrm{Ext}\left( A,SC\right) \rightarrow \mathrm{Ext}\left( SA,SC\right) 
\text{,}
\end{equation*}%
which induce an isomorphism of modules with a Polish cover

\begin{equation*}
\eta _{A,C}:\frac{\mathrm{Ext}\left( A,SC\right) }{\mathrm{Ran}\left( 
\mathrm{Hom}\left( A,C/SC\right) \rightarrow \mathrm{Ext}\left( A,SC\right)
\right) }\rightarrow \mathrm{Ext}\left( SA,SC\right) \text{;}
\end{equation*}

\item The short exact sequence $SC\rightarrow C\rightarrow C/SC$ induces an
exact sequence%
\begin{equation*}
0\rightarrow \mathrm{Hom}\left( A,C/SC\right) \rightarrow \mathrm{Ext}\left(
A,SC\right) \rightarrow S\mathrm{Ext}\left( A,C\right) \rightarrow S\mathrm{%
Ext}\left( A,C/SC\right) \rightarrow 0
\end{equation*}%
which induces a \emph{pure} short exact sequence of modules with a Polish
cover 
\begin{equation*}
0\rightarrow \frac{\mathrm{Ext}\left( A,SC\right) }{\mathrm{Ran}\left( 
\mathrm{Hom}\left( A,C/SC\right) \rightarrow \mathrm{Ext}\left( A,SC\right)
\right) }\overset{\rho _{A,C}}{\rightarrow }S\mathrm{Ext}\left( A,C\right)
\rightarrow S\mathrm{Ext}\left( A,C/SC\right) \rightarrow 0\text{;}
\end{equation*}

\item The injective homomorphism%
\begin{equation*}
r_{A,C}:=\rho _{A,C}\circ \eta _{A,C}^{-1}:\mathrm{Ext}\left( SA,SC\right)
\rightarrow S\mathrm{Ext}\left( A,C\right)
\end{equation*}%
restricts to an isomorphism%
\begin{equation*}
\gamma _{A,C}:\mathrm{PExt}\left( SA,SC\right) \rightarrow u_{1}\left( S%
\mathrm{Ext}\left( A,C\right) \right)
\end{equation*}%
of modules with a Polish cover;

\item We have a \emph{pure} exact sequence of modules with a Polish cover%
\begin{equation*}
0\rightarrow \mathrm{Ext}\left( A^{\alpha },C^{\alpha }\right) \overset{%
r_{A,C}^{\alpha }}{\rightarrow }\mathrm{Ext}\left( A,C\right) ^{\alpha
}\rightarrow \mathrm{Hom}\left( A^{\alpha },E_{\omega \alpha }\left(
C\right) \right) \rightarrow 0\text{;}
\end{equation*}
\end{enumerate}
\end{lemma}

\subsection{Completions\label{Subsection:completion}}

Let $G$ be a module with a Polish cover. For some $\alpha <\omega _{1}$, $%
\alpha $-topology on $G$ is the group topology that has $\left( G^{\beta
}\right) _{\beta <\alpha }$ as neighborhoods of $0$. Let $L_{\alpha }\left(
G\right) :=\mathrm{lim}_{\beta <\alpha }G/G^{\beta }$ be the corresponding
Hausdorff completion, which is a Polish module. We have a canonical exact
sequence of modules with a Polish cover 
\begin{equation*}
0\rightarrow G^{\alpha }\rightarrow G\overset{\kappa _{\alpha }}{\rightarrow 
}L_{\alpha }\left( G\right) \rightarrow E_{\alpha }\left( G\right)
\rightarrow 0\text{.}
\end{equation*}%
The map $\kappa _{\alpha }:G\rightarrow L_{\alpha }\left( G\right) $ is the
homomorphism induced by the quotient maps $G\rightarrow G/G^{\beta }$ for $%
\beta <\alpha $, while 
\begin{equation*}
E_{\alpha }\left( G\right) :=\mathrm{coker}\left( \kappa _{\alpha }\right)
=L_{\alpha }\left( G\right) /\kappa _{\alpha }\left( G\right) \text{.}
\end{equation*}
Notice that the $\alpha $-topology on $G$ is Hausdorff if and only if $%
G^{\alpha }=0$, and it is Hausdorff and complete if and only if $G^{\alpha
}=0$ and $E_{\alpha }\left( G\right) =0$, in which case $G\rightarrow
L_{\alpha }\left( G\right) $ is an isomorphism.

\begin{definition}
\label{Definition:isotype}If $G,H$ are modules with a Polish cover, and $%
\varphi :G\rightarrow H$ is an injective homomorphism, then $\varphi $ is 
\emph{isotype }if $\varphi ^{-1}\left( u_{\alpha }\left( H\right) \right)
=u_{\alpha }\left( G\right) $ for every ordinal $\alpha $. If $G\subseteq H$
is a submodule with a Polish cover, then $G$ is isotype if the inclusion $%
G\rightarrow H$ is isotype.
\end{definition}

If $\varphi :G\rightarrow H$ is an isotype homomorphism of modules with a
Polish cover, then it induces a homomorphism $L_{\alpha }\left( G\right)
\rightarrow L_{\alpha }\left( H\right) $ of Polish modules that makes the
diagram%
\begin{equation*}
\begin{array}{ccc}
G & \rightarrow & H \\ 
\downarrow &  & \downarrow \\ 
L_{\alpha }\left( G\right) & \rightarrow & L_{\alpha }\left( H\right)%
\end{array}%
\end{equation*}%
commute. In turn, this induces a homomorphism $E_{\alpha }\left( G\right)
\rightarrow E_{\alpha }\left( H\right) $ of modules with a Polish cover.

\subsection{Ulm cotorsion functors}

For every ordinal $\alpha <\omega _{1}$, the assignment $A\mapsto p^{\alpha
}A$ defines on countable torsion modules a cotorsion functor with enough
projectives \cite[Chapter V]{griffith_infinite_1970}. For example, $A\mapsto
p^{\omega }A$ is the cotorsion functor represented by the extension $%
R\rightarrow R[1/p]\rightarrow R\left( p^{\infty }\right) $. More generally,
for $\alpha <\omega _{1}$, $A\mapsto A^{\alpha }$ is represented by the
extension $R\rightarrow G_{\alpha }\rightarrow H_{\alpha }$ as in \cite[%
Theorem 69]{griffith_infinite_1970}, where $H_{\alpha }$ is a countable
module of Ulm length $\alpha $. A torsion module $A$ is called \emph{totally
projective }if and only if $A/p^{\alpha }A$ is $p^{\alpha }$-projective for
every ordinal $\alpha $ \cite[Section 82]{fuchs_infinite_1973}. Every
countable torsion module is totally projective \cite[Theorem 82.4]%
{fuchs_infinite_1973}. For a countable module $C$ and limit ordinal $\alpha
<\omega _{1}$, define as in the previous section 
\begin{equation*}
L_{\alpha }\left( C\right) =\mathrm{lim}_{\beta <\alpha }C/p^{\beta }C\text{,%
}
\end{equation*}%
and consider the Borel-definable exact sequence 
\begin{equation*}
0\rightarrow p^{\alpha }C\rightarrow C\rightarrow L_{\alpha }\left( C\right)
\rightarrow E_{\alpha }\left( C\right) \rightarrow 0\text{.}
\end{equation*}

\begin{lemma}
\label{Lemma:Keef}Suppose that $A,C$ are countable $p$-groups, and $\alpha
<\omega _{1}$ is a limit ordinal. Then $\mathrm{Hom}\left( A,E_{\alpha
}\left( C\right) \right) $ and $\mathrm{Hom}\left( A,p^{\alpha }\mathrm{Ext}%
\left( G_{\alpha },C\right) \right) $ are isomorphic groups with a Polish
cover.
\end{lemma}

\begin{proof}
By Lemma \ref{Lemma:short-exact-functor}(1), after replacing $C$ with $%
C/p^{\alpha }C$, we can assume that $p^{\alpha }C=0$. Consider the short
exact sequence 
\begin{equation*}
C\rightarrow \mathrm{Ext}\left( H_{\alpha },C\right) \rightarrow \mathrm{Ext}%
\left( G_{\alpha },C\right)
\end{equation*}%
of modules with a Polish cover induced by 
\begin{equation*}
R\rightarrow G_{\alpha }\rightarrow H_{\alpha }\text{.}
\end{equation*}%
As in the proof of \cite{keef_injective_1994}, $\mathrm{Ext}\left( H_{\alpha
},C\right) $ is complete in the $\alpha $-topology. Furthermore, the
monomorphism $C\rightarrow \mathrm{Ext}\left( H_{\alpha },C\right) $ is
isotype as in Definition \ref{Definition:isotype}, and hence it induces a
monomorphism of Polish modules 
\begin{equation*}
L_{\alpha }\left( C\right) \rightarrow \mathrm{Ext}\left( H_{\alpha
},C\right) \cong L_{\alpha }\left( \mathrm{Ext}\left( H_{\alpha },C\right)
\right) \text{.}
\end{equation*}%
In turn, this induces an injective homomorphism of modules with a Polish
cover 
\begin{equation*}
E_{\alpha }\left( C\right) \rightarrow \mathrm{Ext}\left( G_{\alpha
},C\right) \text{.}
\end{equation*}%
It is proved in \cite{keef_injective_1994} that this homomorphism induces an
isomorphism between the torsion submodules of $E_{\alpha }\left( C\right) $
and $p^{\alpha }\mathrm{Ext}\left( G_{\alpha },C\right) $. As $A$ is a
torsion module, this induces an isomorphism of modules with a Polish cover 
\begin{equation*}
\mathrm{Hom}\left( A,E_{\alpha }\left( C\right) \right) \cong \mathrm{Hom}%
\left( A,p^{\alpha }\mathrm{Ext}\left( G_{\alpha },C\right) \right) \text{.}
\end{equation*}%
This concludes the proof.
\end{proof}

The following corollary is an immediate consequence of Lemma \ref{Lemma:Keef}
and Lemma \ref{Lemma:short-exact-functor}, considering that every countable
torsion module is totally projective. Recall that, for a module $A$ and $%
\alpha <\omega _{1}$, $A^{\alpha }=p^{\omega \alpha }A$.

\begin{corollary}
\label{Corollary:Keef}Suppose that $A,C$ are countable torsion modules, and $%
\alpha <\omega _{1}$ is an ordinal.

\begin{enumerate}
\item The quotient map $C\rightarrow C/C^{\alpha }$ induces isomorphisms 
\begin{equation*}
\mathrm{Ext}\left( G_{\omega \alpha },C\right) \rightarrow \mathrm{Ext}%
\left( G_{\omega \alpha },C/C^{\alpha }\right)
\end{equation*}%
and 
\begin{equation*}
\mathrm{Ext}\left( H_{\omega \alpha },C\right) \rightarrow \mathrm{Ext}%
\left( H_{\omega \alpha },C/C^{\alpha }\right) \text{;}
\end{equation*}

\item If $C^{\alpha }=0$ then $\mathrm{Ext}\left( A,C\right) ^{\alpha }$ and 
$\mathrm{Hom}\left( A^{\alpha },E_{\omega \alpha }\left( C\right) \right) $
are isomorphic Polish modules;

\item The short exact sequences $A^{\alpha }\rightarrow A\rightarrow
A/A^{\alpha }$ and $C^{\alpha }\rightarrow C\rightarrow C/C^{\alpha }$
induce homomorphisms 
\begin{equation*}
\mathrm{Ext}\left( A,C^{\alpha }\right) \rightarrow \mathrm{Ext}\left(
A^{\alpha },C^{\alpha }\right)
\end{equation*}%
and 
\begin{equation*}
\mathrm{Hom}\left( A,C/C^{\alpha }\right) \rightarrow \mathrm{Ext}\left(
A,C^{\alpha }\right) \text{,}
\end{equation*}%
which induce an isomorphism%
\begin{equation*}
\eta _{A,C}^{\alpha }:\frac{\mathrm{Ext}\left( A,C^{\alpha }\right) }{%
\mathrm{Ran}\left( \mathrm{Hom}\left( A,C/C^{\alpha }\right) \rightarrow 
\mathrm{Ext}\left( A,C^{\alpha }\right) \right) }\rightarrow \mathrm{Ext}%
\left( A^{\alpha },C^{\alpha }\right) \text{;}
\end{equation*}

\item The short exact sequence $C^{\alpha }\rightarrow C\rightarrow
C/C^{\alpha }$ induces an exact sequence%
\begin{equation*}
0\rightarrow \mathrm{Hom}\left( A,C/C^{\alpha }\right) \rightarrow \mathrm{%
Ext}\left( A,C^{\alpha }\right) \rightarrow \mathrm{Ext}\left( A,C\right)
^{\alpha }\rightarrow \mathrm{Ext}\left( A,C/C^{\alpha }\right) ^{\alpha
}\rightarrow 0
\end{equation*}%
which induces a pure exact sequence%
\begin{equation*}
0\rightarrow \frac{\mathrm{Ext}\left( A,C^{\alpha }\right) }{\mathrm{Ran}%
\left( \mathrm{Hom}\left( A,C/C^{\alpha }\right) \rightarrow \mathrm{Ext}%
\left( A,C^{\alpha }\right) \right) }\overset{\rho _{A,C}^{\alpha }}{%
\rightarrow }\mathrm{Ext}\left( A,C\right) ^{\alpha }\rightarrow \mathrm{Ext}%
\left( A,C/C^{\alpha }\right) ^{\alpha }\rightarrow 0\text{;}
\end{equation*}

\item The injective homomorphism%
\begin{equation*}
r_{A,C}^{\alpha }:=\rho _{A,C}^{\alpha }\circ (\eta _{A,C}^{\alpha })^{-1}:%
\mathrm{Ext}\left( A^{\alpha },C^{\alpha }\right) \rightarrow \mathrm{Ext}%
\left( A,C\right) ^{\alpha }
\end{equation*}%
restricts to an isomorphism%
\begin{equation*}
\gamma _{A,C}^{\alpha }:\mathrm{PExt}\left( A^{\alpha },C^{\alpha }\right)
\rightarrow \mathrm{Ext}\left( A,C\right) ^{\alpha +1}\text{;}
\end{equation*}

\item We have a \emph{pure} exact sequence of modules with a Polish cover%
\begin{equation*}
0\rightarrow \mathrm{Ext}\left( A^{\alpha },C^{\alpha }\right) \overset{%
r_{A,C}^{\alpha }}{\rightarrow }\mathrm{Ext}\left( A,C\right) ^{\alpha
}\rightarrow \mathrm{Hom}\left( A^{\alpha },E_{\omega \alpha }\left(
C\right) \right) \rightarrow 0
\end{equation*}
\end{enumerate}
\end{corollary}

\section{Phantom extensions\label{Section:phantom}}

In this section, we continue to assume that $R\ $is a countable DVR, and all
modules are $R$-modules. We characterize for a given countable ordinal $%
\alpha $ the extensions of torsion modules that are phantom of order $\alpha 
$. Furthermore, we determine the (plain) Solecki length of $\mathrm{Ext}%
\left( C,A\right) $ for given torsion modules $C$ and $A$ in terms of their
Ulm invariants.

\subsection{Complexity of $\mathrm{Hom}$}

We being with recording some lemmas concerning the complexity of the module
with a Polish cover $E_{\alpha }\left( A\right) $ as defined in Section \ref%
{Subsection:completion}.

\begin{lemma}
\label{Lemma:Ealpha}Suppose that $A$ is a countable reduced torsion module, $%
d\in \omega $, and $\alpha <\omega _{1}$ is a limit ordinal such that $%
p^{\beta }A\neq 0$ for every $\beta <\alpha $. Then:

\begin{itemize}
\item $E_{\alpha }\left( A\right) $ is divisible;

\item $E_{\alpha }\left( A\right) [p^{d}]=\left\{ x\in E_{\alpha }\left(
A\right) :p^{d}x=0\right\} $ is a submodule with a Polish cover of $%
E_{\alpha }\left( A\right) $;

\item $E_{\alpha }\left( A\right) [p^{d}]$ and $E_{\alpha }\left( A\right) $
have plain Solecki length $1$.
\end{itemize}
\end{lemma}

\begin{proof}
After replacing $A$ with $A/p^{\alpha }A$ we can assume that $p^{\alpha }A=0$%
. Thus, the canonical map $A\rightarrow L_{\alpha }\left( A\right) $ is a
monomorphism, and $E_{\alpha }\left( A\right) =L_{\alpha }\left( A\right) /A$%
. Notice that $pL_{\alpha }\left( A\right) $ is an open submodule of $%
L_{\alpha }\left( A\right) $, and $A$ is a dense submodule $L_{\alpha
}\left( A\right) $. Thus, $pL_{\alpha }\left( A\right) +A=L_{\alpha }\left(
A\right) $.\ Hence, the map $pL_{\alpha }\left( A\right) \rightarrow
E_{\alpha }\left( A\right) $, $a\mapsto a+A$ is surjective, and $E_{\alpha
}\left( A\right) $ is divisible. Since $p^{\beta }A\neq 0$ for every $\beta
<\alpha $, we have that $L_{\alpha }\left( A\right) $ is uncountable. In
particular, $E_{\alpha }\left( A\right) $ and $E_{\alpha }\left( A\right)
[p^{d}]$ are nontrivial. Fix a strictly increasing sequence $\left( \alpha
_{n}\right) _{n\in \omega }$ of countable ordinals such that $\alpha _{0}=0$
and $\mathrm{sup}_{n}\alpha _{n}=\alpha $. We have that $E_{\alpha }\left(
A\right) [p^{d}]=\hat{G}/A$ where 
\begin{equation*}
\hat{G}:=\left\{ x\in L_{\alpha }\left( A\right) :p^{d}x\in A\right\}
\subseteq L_{\alpha }\left( A\right) \text{.}
\end{equation*}%
The Polish module topology on $\hat{G}$ is defined by letting $\left(
x_{i}\right) $ converge to $0$ if and only if $x_{i}\rightarrow x$ in $%
L_{\alpha }\left( A\right) $ and $p^{d}x_{i}=p^{d}x$ eventually. Thus, $%
E_{\alpha }\left( A\right) [p^{d}]$ is a submodule with a Polish cover of $%
E_{\alpha }\left( A\right) $. This also follows from the fact that $%
E_{\alpha }\left( A\right) [p^{d}]$ is the kernel of the homomorphism $%
E_{\alpha }\left( A\right) \rightarrow E_{\alpha }\left( A\right) $, $%
x\mapsto p^{d}x$.

It is clear that $\left\{ 0\right\} $ is $\boldsymbol{\Sigma }_{2}^{0}$ in $%
E_{\alpha }\left( A\right) $, since $A$ is countable. We now show that $%
\left\{ 0\right\} $ is dense in $E_{\alpha }\left( A\right) [p^{d}]$.

Consider an element $x$ of $E_{\alpha }\left( A\right) [p^{d}]$. Then $x$
can be written as $a+A$ where $a\in L_{\alpha }\left( A\right) $. In turn,
one can write $a$ as $\sum_{n\in \omega }a_{n}$ for $a_{n}\in p^{\alpha
_{n}}A$. Since $p^{d}x=0$ we have that $\sum_{n}p^{d}a_{n}\in A\cap
p^{d}L_{\alpha }\left( A\right) =p^{d}A$.\ Thus, we can find $b\in A$ such
that $\sum_{n}p^{d}a_{n}=p^{d}b$. After replacing $a_{0}$ with $a_{0}-b$ we
can assume that $\sum_{n}p^{d}a_{n}=0$. Thus, for every $k\in \omega $, we
have that $\sum_{n=0}^{k}p^{d}a_{n}\in p^{\alpha _{k+1}+d}A$ and hence we
can find $b_{k}\in p^{\alpha _{k+1}}A$ such that $%
\sum_{n=0}^{k}p^{d}a_{n}=p^{d}b_{k}$. Thus, we have that $\left(
a_{0}+a_{1}+\cdots +a_{k}-b_{k}\right) _{k\in \omega }$ is a sequence in $%
\hat{G}$ that converges to $a$. This shows that $\left\{ 0\right\} $ is
dense in $E_{\alpha }\left( A\right) [p^{d}]$, and hence not closed in $%
E_{\alpha }\left( A\right) [p^{d}]$. This concludes the proof that $%
E_{\alpha }\left( A\right) [p^{d}]$ and $E_{\alpha }\left( A\right) $ have
plain Solecki length $1$.
\end{proof}

\begin{proposition}
\label{Proposition:Hom-Ealpha}Suppose that $T$ is a nonzero countable
module, $A$ is a countable reduced torsion module, and $\alpha <\omega _{1}$
is a limit ordinal such that $p^{\beta }A\neq 0$ for every $\beta <\alpha $.

\begin{enumerate}
\item If $T$ is finite, then $\boldsymbol{\Sigma }_{2}^{0}$ is the
complexity class of $\left\{ 0\right\} $ in $\mathrm{Hom}\left( T,E_{\omega
\alpha }\left( A\right) \right) $.

\item If $T$ is not finite, then $\boldsymbol{\Pi }_{3}^{0}$ is the
complexity class of $\left\{ 0\right\} $ in $\mathrm{Hom}\left( T,E_{\omega
\alpha }\left( A\right) \right) $.
\end{enumerate}
\end{proposition}

\begin{proof}
Suppose initially that $T$ is finite. In this case, $T$ is a finite sum of
cyclic modules. Thus, without loss of generality, we can assume that $T$ is
cyclic of order $p^{d}$ for some $d\geq 1$. In this case, we have that $%
\mathrm{Hom}\left( T,E_{\alpha }\left( A\right) \right) $ is isomorphic to $%
E_{\omega \alpha }\left( A\right) [p^{d}]$, and the conclusion follows from
Lemma \ref{Lemma:Ealpha}.

Suppose now $T$ is not finite, and reduced. The quotient map $T\rightarrow
T/T^{1}$ induces an injective homomorphism%
\begin{equation*}
\mathrm{Hom}\left( T/T^{1},E_{\omega \alpha }\left( A\right) \right)
\rightarrow \mathrm{Hom}\left( T,E_{\omega \alpha }\left( A\right) \right) 
\text{.}
\end{equation*}%
Then after replacing $T$ with $T/T^{1}$, we can assume that $T$ has Ulm rank 
$1$, and hence $T\cong \bigoplus_{n}T_{n}$ where, for every $n\in \omega $, $%
T_{n}$ is a nonzero cyclic torsion module. In this case, we have that%
\begin{equation*}
\mathrm{Hom}\left( T,E_{\omega \alpha }\left( A\right) \right) \cong
\prod_{n\in \omega }\mathrm{Hom}\left( T_{n},E_{\omega \alpha }\left(
A\right) \right) \text{.}
\end{equation*}%
The conclusion now follows from the case when $T$ is finite.

Suppose lastly that $T$ is not reduced. Without loss of generality, it
suffices to consider the case when $T=R\left( p^{\infty }\right) $.\ Then we
have that $\mathrm{Hom}\left( T,E_{\omega \alpha }\left( A\right) \right) $
is isomorphic to the limit of the tower whose terms are $E_{\omega \alpha
}\left( A\right) $ and bonding maps $E_{\omega \alpha }\left( A\right)
\rightarrow E_{\omega \alpha }\left( A\right) $, $x\mapsto px$. Since 
\begin{equation*}
E_{\omega \left( 1+\alpha -1\right) }\left( A\right) [p]\neq 0\text{,}
\end{equation*}
$\mathrm{Hom}\left( T,E_{\omega \alpha }\left( A\right) \right) $ is not
plain by \cite[Corollary 5.9]{casarosa_projective_2025}.
\end{proof}

\subsection{Solecki length of $\mathrm{Ext}$}

We now determine the (plain) Solecki length of $\mathrm{Ext}\left(
T,A\right) $ for given countable torsion modules $C,A$.

\begin{lemma}
\label{Lemma:vanishing-PExt}Suppose that $A,C$ are countable torsion
modules, and $\alpha <\omega _{1}$ is a countable ordinal.

\begin{enumerate}
\item If $\alpha $ is a successor, then%
\begin{equation*}
\mathrm{PExt}\left( A^{1+\alpha -1},C^{1+\alpha -1}\right) \cong s_{\alpha }%
\mathrm{Ext}\left( A,C\right) \text{;}
\end{equation*}

\item if 
\begin{equation*}
\mathrm{Ext}\left( A^{1+\alpha },C^{1+\alpha }\right) =0
\end{equation*}%
then%
\begin{equation*}
s_{\alpha }\mathrm{Ext}\left( A,C\right) \cong u_{1+\alpha }\mathrm{Ext}%
\left( A,C\right) \cong \mathrm{Hom}\left( A^{1+\alpha },E_{\omega \left(
1+\alpha \right) }C\right) \text{;}
\end{equation*}

\item if $A^{1+\alpha }=0$, then $s_{\alpha }\mathrm{Ext}\left( A,C\right)
=0 $;

\item if $\alpha $ is a successor and $C^{1+\alpha -1}$ is bounded, then $%
s_{\alpha }\mathrm{Ext}\left( A,C\right) =0$.
\end{enumerate}
\end{lemma}

\begin{proof}
The first two items follow from Corollary \ref{Corollary:Keef}(5) and \ref%
{Corollary:Keef}(6).

(3) If $A^{1+\alpha }=0$ then $\mathrm{Ext}\left( A^{1+\alpha },C^{1+\alpha
}\right) =0$ and $\mathrm{Hom}\left( A^{1+\alpha },E_{\omega \left( 1+\alpha
\right) }C\right) =0$. Thus the conclusion follows from (2).

(4)\ If $C^{1+\alpha -1}$ is bounded, then $C^{1+\alpha }=0$, $\mathrm{Ext}%
\left( A^{1+\alpha },C^{1+\alpha }\right) =0$ and $\mathrm{Hom}\left(
A^{1+\alpha },E_{\omega \left( 1+\alpha \right) }C\right) =0$. Thus, the
conclusion again follows from (2).
\end{proof}

\begin{theorem}
\label{Theorem:Solecki-Ext}Suppose that $C$ and $A$ are countable torsion
modules, with $A$ reduced. Let $\alpha $ be the least countable ordinal such
that either $C^{1+\alpha }=0$, or $\alpha $ is zero or successor and $%
A^{1+\alpha -1}$ is bounded. Then $\mathrm{Ext}\left( C,A\right) $ has
Solecki length $\alpha $. When $\alpha $ is a successor, we have that $%
\mathrm{Ext}\left( C,A\right) $ is plain if and only if $C^{1+\alpha -1}$ is
finite.
\end{theorem}

\begin{proof}
It follows from the Lemma \ref{Lemma:vanishing-PExt} that $s_{\alpha }%
\mathrm{Ext}\left( C,A\right) =0$, and $\mathrm{Ext}\left( C,A\right) $ has
Solecki length at most $\alpha $.

If $\alpha $ is limit, then we have for $\beta <\alpha $ successor%
\begin{equation*}
s_{\beta }\mathrm{Ext}\left( C,A\right) \cong \mathrm{PExt}\left( C^{1+\beta
-1},A^{1+\beta -1}\right) \neq 0
\end{equation*}%
by Lemma \ref{Lemma:vanishing-PExt}. This shows that $\mathrm{Ext}\left(
C,A\right) $ has Solecki length $\alpha $.

If $\alpha $ is a successor, then by Lemma \ref{Lemma:vanishing-PExt}%
\begin{equation*}
s_{\alpha -1}\mathrm{Ext}\left( C,A\right) \cong \mathrm{Hom}\left(
C^{1+\alpha -1},E_{\omega \left( 1+\alpha -1\right) }\left( A\right) \right)
\neq 0\text{.}
\end{equation*}%
This shows that $\mathrm{Ext}\left( T,A\right) $ has Solecki length equal to 
$\alpha $, and it is plain if and only if $C^{1+\alpha -1}$ is finite by
Proposition \ref{Proposition:Hom-Ealpha}.
\end{proof}

Combining Theorem \ref{Theorem:Solecki-Ext} in the case when $C$ is
divisible, together with Corollary \ref{Corollary:Keef}(5) and Theorem \ref%
{Theorem:ulm} yields the following.

\begin{corollary}
\label{Corollary:Solecki-Ext}Suppose that $A$ is a countable reduced torsion
module. Let $\alpha $ be the least countable successor ordinal such that $%
A^{1+\alpha -1}$ is bounded. Then $\mathrm{Ext}\left( R\left( p^{\infty
}\right) ,A\right) $ has Solecki length $\alpha $. Furthermore, for every
successor ordinal $\beta <\alpha $, the inclusion $A^{\beta }\rightarrow A$
induces an isomorphism $\mathrm{PExt}\left( R\left( p^{\infty }\right)
,A^{1+\beta -1}\right) \cong s_{\beta }\mathrm{Ext}\left( R\left( p^{\infty
}\right) ,A\right) =u_{1+\beta }\mathrm{Ext}\left( R\left( p^{\infty
}\right) ,A\right) $.
\end{corollary}

We can record the following complexity-theoretic consequence of Theorem \ref%
{Theorem:Solecki-Ext} obtained by applying \cite[Theorem 12.12]%
{casarosa_projective_2025} and \cite[Theorem 12.21]{casarosa_projective_2025}%
; see \cite[Section 12]{casarosa_projective_2025} for notation and
terminology.

\begin{corollary}
\label{Corollary:parametrize-Ext}Let $C$ and $A$ be countable torsion
modules, with $A$ reduced. Let $\alpha $ be the least countable ordinal such
that either $C^{1+\alpha }=0$, or $\alpha $ is zero or successor and $%
A^{1+\alpha -1}$ is bounded. Then extensions of $C$ by $A$ can be
parametrized by hereditarily countable sets of rank $\alpha $, but not by
hereditarily countable sets of rank $\beta <\alpha $. Furthermore, when $%
\alpha $ is a successor, extensions of $C$ by $A$ can be parametrized by
hereditarily countable sets of plain rank $\alpha $ if and only if $%
C^{1+\alpha -1}$ is finite.
\end{corollary}

From Theorem \ref{Theorem:Solecki-Ext} and Theorem \ref{Theorem:ulm} we
obtain the following. Recall the definition of the\ countable well-founded
rooted tree $I_{\alpha }^{\mathrm{plain}}$ and forest $I_{\alpha }$ of rank $%
\alpha $ from Section \ref{Subsection:projective-length}.

\begin{corollary}
\label{Corollary:projective}Suppose that $C$ is a countable torsion module,
and $\alpha <\omega _{1}$, the following assertions are equivalent:

\begin{enumerate}
\item $C$ has plain projective length at most $\alpha $;

\item $C$ is reduced of Ulm length at most $1+\alpha $ and $C^{1+\alpha -1}$
is bounded;
\end{enumerate}

Furthermore, the following assertions are equivalent:

\begin{enumerate}
\item $C$ has projective length at most $\alpha $;

\item $C$ is reduced of Ulm length at most $1+\alpha $;

\item $C$ is a colimit of a presheaf of finite torsion modules over a
countable well-founded forest of rank $1+\alpha $.
\end{enumerate}
\end{corollary}

\begin{proof}
The equivalence of (1) and (2) follows from Theorem \ref{Theorem:Solecki-Ext}
and Theorem \ref{Theorem:ulm}. By Proposition \ref%
{Proposition:projective-length} we have (3)$\Rightarrow $(1). We prove that
(2)$\Rightarrow $(3) by induction on $\alpha $. For $\alpha =0$, we have
that a reduced countable torsion module of Ulm length at most $1$ is a
countable direct sum of cyclic modules, and the conclusion follows. For $%
\alpha $ limit, we have that a countable torsion module of Ulm length at
most $\alpha $ is a countable direct sum of modules of Ulm length at most $%
\alpha _{n}$ for $n<\omega $, and the conclusion follows from the inductive
hypothesis. Suppose that $\alpha $ is a successor and $C$ has Ulm length $%
\alpha $. Then by the Ulm Classification Theorem \cite[Theorem 77.3,
Corollary 76.2]{fuchs_infinite_1973} we can assume without loss of
generality that $u_{\alpha -1}C$ is finite. By the inductive hypothesis, $%
C/u_{\alpha -1}C$ is a colimit of a presheaf $\left( F_{i}\right) _{i\in
I_{\alpha -1}}$ of finite torsion modules over $I_{1+\alpha -1}$. From $%
u_{\alpha -1}C$ and $\left( F_{i}\right) _{i\in I_{1+\alpha -1}}$ one easily
produces a presheaf of finite torsion modules over $I_{1+\alpha }^{\mathrm{%
plain}}$ with colimit $C$.
\end{proof}

Notice that by virtue of Theorem \ref{Theorem:ulm}, the projective length of 
$C$ is at most $\alpha $ if and only if $C$ is $p^{\omega \alpha }$%
-projective in the sense of \cite{keef_injective_1994}. Thus, Corollary \ref%
{Corollary:projective} recovers a classical result of Nunke from \cite%
{nunke_purity_1963}; see also \cite[Chapter XII, page 92, Exercise 2]%
{fuchs_infinite_1973}.

\subsection{Reduction to the local case}

In this last section, we assume that $R$ is a countable Dedekind domain, and
assume all modules to be $R$-modules. We show that in the study of (higher
order) pure extensions of torsion modules, one can reduce to the local case
by localizing with respect to prime ideals. We let $\mathbb{P}$ be the set
of nonzero prime ideals of $R$. For a countable module $C$, we let%
\begin{equation*}
C_{\mathfrak{p}}:=\left\{ x\in C:\mathfrak{p}x=0\right\}
\end{equation*}%
be its $\mathfrak{p}$-primary submodule. It is easily seen using the prime
factorization of ideals in Dedekind domains that every torsion module is the
direct sum of its $\mathfrak{p}$-primary submodules for $\mathfrak{p}\in 
\mathbb{P}$. A module $A$ is $\mathfrak{p}$-divisible if $\mathfrak{p}A=A$.

\begin{lemma}
\label{Lemma:p,q-zero}Fix $\mathfrak{p}\in \mathbb{P}$. Suppose that $A,C$
are countable modules such that $C$ is $\mathfrak{p}$-primary and $A$ is $%
\mathfrak{p}$-divisible. Then $\mathrm{Ext}\left( C,A\right) =0$.
\end{lemma}

\begin{proof}
Let $D$ be the divisible hull of $A$ \cite[Section 24]{fuchs_infinite_1970}.
Then $D/A$ is a torsion module with trivial $\mathfrak{p}$-primary
component. In particular, $\mathrm{Hom}\left( C,D/A\right) =0$. As $D$ is
divisible, $\mathrm{Ext}\left( C,D\right) =0$. Considering the exact sequence%
\begin{equation*}
\mathrm{Hom}\left( C,D/A\right) \rightarrow \mathrm{Ext}\left( C,A\right)
\rightarrow \mathrm{Ext}\left( C,D\right)
\end{equation*}%
induced by the exact sequence $A\rightarrow D\rightarrow D/A$, we conclude
that $\mathrm{Ext}\left( C,A\right) =0$.
\end{proof}

\begin{lemma}
\label{Lemma:reduce-p}Suppose that $C$ is a countable $\mathfrak{p}$-primary
module and $A$ is a countable torsion module. Then the inclusion $A_{%
\mathfrak{p}}\rightarrow A$ induces an isomorphism $\mathrm{PExt}\left( C,A_{%
\mathfrak{p}}\right) \cong \mathrm{PExt}\left( C,A\right) $.
\end{lemma}

\begin{proof}
We can write $A=A_{\mathfrak{p}}\oplus B$ where $B$ is $\mathfrak{p}$%
-divisible. By Lemma \ref{Lemma:p,q-zero}, $\mathrm{PExt}\left( C,B\right)
=0 $. Therefore, $\mathrm{PExt}\left( C,A\right) $ is isomorphic to $\mathrm{%
PExt}\left( C,A_{\mathfrak{p}}\right) \oplus \mathrm{PExt}\left( C,B\right) =%
\mathrm{PExt}\left( C,A_{\mathfrak{p}}\right) $.
\end{proof}

\begin{proposition}
\label{Proposition:reduce-p}Let $R$ be a countable Dedekind domain. Suppose
that $C$ is a countable torsion module and $A$ is a countable reduced
module. Then $\mathrm{PExt}\left( C,A\right) $ and%
\begin{equation*}
\prod_{\mathfrak{p}\in \mathbb{P}}\mathrm{PExt}\left( C_{\mathfrak{p}},A_{%
\mathfrak{p}}\right)
\end{equation*}%
are isomorphic modules with a Polish cover.
\end{proposition}

\begin{proof}
As observed in Section \ref{Section:pure}, we can assume without loss of
generality that $A$ is a torsion module. Since $C$ is a countable torsion
module, we have that $C\cong \bigoplus_{\mathfrak{p}}C_{\mathfrak{p}}$.\
Therefore, $\mathrm{PExt}\left( C,A\right) $ is isomorphic to $\prod_{%
\mathfrak{p}}\mathrm{PExt}\left( C_{\mathfrak{p}},A\right) $. By Lemma \ref%
{Lemma:reduce-p}, for every prime $p$ we have that $\mathrm{PExt}\left( C_{%
\mathfrak{p}},A\right) $ is isomorphic to $\mathrm{PExt}\left( C_{\mathfrak{p%
}},A_{\mathfrak{p}}\right) $. This concludes the proof.
\end{proof}

\bibliographystyle{amsalpha}
\bibliography{bibliography}
\printindex

\end{document}